\documentclass{article}
\usepackage{latexsym,amsfonts,amsmath,amsthm,amssymb,makeidx}
\usepackage[title]{appendix}
\usepackage{CJK,CJKnumb,CJKulem,times,dsfont,ifthen,mathrsfs,latexsym,amsfonts, color}
\usepackage{amsmath,amsthm,makeidx,fontenc,amssymb,bm,graphicx,psfrag,listings, curves,extarrows}
\usepackage{cite}

\let\oldbibliography\thebibliography
\renewcommand{\thebibliography}[1]{%
\oldbibliography{#1}%
\setlength{\itemsep}{0pt}%
}

\makeindex
\newtheorem{definition}{Definition}[section]
\newtheorem{theorem}{Theorem}[section]
\newtheorem{lemma}{Lemma}[section]
\newtheorem{corollary}{Corollary}[section]
\newtheorem{proposition}{Proposition}[section]
\newtheorem{remark}{Remark}[section]

\newcommand{\bt}{\begin{theorem}}
\newcommand{\et}{\end{theorem}}
\newcommand{\bl}{\begin{lemma}}
\newcommand{\el}{\end{lemma}}
\newcommand{\bd}{\begin{definition}}
\newcommand{\ed}{\end{definition}}
\newcommand{\bc}{\begin{corollary}}
\newcommand{\ec}{\end{corollary}}
\newcommand{\bp}{\begin{proof}}
\newcommand{\ep}{\end{proof}}
\newcommand{\bx}{\begin{example}}
\newcommand{\ex}{\end{example}}
\newcommand{\bi}{\begin{exercise}}
\newcommand{\ei}{\end{exercise}}
\newcommand{\bo}{\begin{prop}}
\newcommand{\eo}{\end{prop}}
\newcommand{\br}{\begin{remark}}
\newcommand{\er}{\end{remark}}
\newcommand{\be}{\begin{equation}}
\newcommand{\ee}{\end{equation}}
\newcommand{\ba}{\begin{align}}
\newcommand{\ea}{\end{align}}
\newcommand{\bn}{\begin{enumerate}}
\newcommand{\en}{\end{enumerate}}
\newcommand{\bg}{\begin{align*}}
\newcommand{\bcs}{\begin{cases}}
\newcommand{\ecs}{\end{cases}}

\newcommand{\sg}{\sigma}
\newcommand{\bean}{\begin{eqnarray*}}
\newcommand{\eean}{\end{eqnarray*}}

\numberwithin{equation}{section}

\begin{document}
\title{{\bf  Exact Asymptotic Behavior of Singular Positive Solutions  of Fractional Semi-Linear Elliptic  Equations}\thanks {Supported by  NSFC.   \qquad E-mail addresses:  hui-yang15@mails.tsinghua.edu.cn (H.  Yang),   \qquad \qquad \qquad  wzou@math.tsinghua.edu.cn (W.  Zou)}
}

\date{}
\author{\\{\bf Hui  Yang$^{1}$,\;\;  Wenming Zou$^{2}$}\\
\footnotesize {\it  $^{1}$Yau Mathematical Sciences Center, Tsinghua University, Beijing 100084, China}\\
\footnotesize {\it  $^{2}$Department of Mathematical Sciences, Tsinghua University, Beijing 100084, China}
}

\maketitle
\begin{center}
\begin{minipage}{120mm}
\begin{center}{\bf Abstract}\end{center}

In this paper, we prove the exact asymptotic behavior of singular positive solutions of  fractional semi-linear equations 
$$ (-\Delta)^\sigma u = u^p~~~~~~~\textmd{in} ~ B_1\backslash \{0\}$$   
with an isolated singularity, where $\sigma \in (0,  1)$ and $\frac{n}{n-2\sigma}  <  p  < \frac{n+2\sigma}{n-2\sigma}$.

\vskip0.10in


\end{minipage}

\end{center}

\vskip0.390in

\section{Introduction and Main results}
In this paper, we shall describe the exact asymptotic behavior of singular positive solutions of 
\begin{equation}\label{Iso1}
(-\Delta)^\sg u= u^p~~~~~~~~~~~\textmd{in}~ B_1 \backslash  \{0\}
\end{equation}
with an isolated singularity at the origin, where the punctured unit ball $B_1 \backslash \{0\} \subset \mathbb{R}^n$ with $n \geq2 $, 
$\sg \in (0, 1)$ and $\frac{n}{n-2\sigma}  <  p  < \frac{n+2\sigma}{n-2\sigma}$. $(-\Delta)^\sg$ is the fractional Laplacian.

In \cite{Y-Z}, we  classify the isolated singularities of equation \eqref{Iso1} with $\frac{n}{n-2\sigma}  <  p  < \frac{n+2\sigma}{n-2\sigma}$. More precisely, let $u$ be a nonnegative solution of \eqref{Iso1},
then either  the singularity near 0 is removable, or there exist two positive constants $c_1$ and $c_2$ such that
\begin{equation}\label{Be00}
c_1 |x|^{-\frac{2\sg}{p-1}} \leq u(x) \leq c_2 |x|^{-\frac{2\sg}{p-1}}.
\end{equation}
Here we will prove the  exact asymptotic behavior of singular positive solutions in \eqref{Be00}. Our main result is the following 
\begin{theorem}\label{THM01}
Let $u$ be a nonnegative solution of \eqref{Iso1}. Assume 
$$\frac{n}{n-2\sg} < p < \frac{n+2\sg}{n-2\sg}.$$
Then either  the singularity near 0 is removable, or
\begin{equation}\label{Be01}
\lim_{|x|\rightarrow 0} |x|^{\frac{2\sigma}{p-1}} u(x) = A_{n,p,\sigma},
\end{equation}
where 
\begin{equation}\label{A}
A_{n,p,\sigma}=\left\{\Lambda\left(\frac{n-2\sg}{2} - \frac{2\sg}{p-1}\right)\right\}^{\frac{1}{p-1}}
\end{equation}
with
$$
\Lambda(\alpha)=2^{2\sg} \frac{\Gamma(\frac{n+2\sg+2\alpha}{4})   \Gamma(\frac{n+2\sg-2\alpha}{4})}{\Gamma(\frac{n-2\sg-2\alpha}{4})\Gamma(\frac{n-2\sg+2\alpha}{4})}.
$$
\end{theorem}
For the classical case $\sg=1$, Theorem \ref{THM01} has been proved in the pioneering  paper \cite{G-S} by  Gidas and Spruck.   We may also see another proof in classical paper \cite{C-G-S} by Caffarelli, Gidas and Spruck.  We remark that in these two proofs, the ODEs analysis play an essential role. However, the ODEs technique is a missing ingredient in our fractional case. Hence, we need some new ideas to overcome this key difficulty.  Here we shall use a monotonicity formula established in our recent paper \cite{Y-Z} and a blow up argument introduced  in Ghergu-Kim-Shahgholian \cite{G-K-S} to solve this problem.

We study the equation \eqref{Iso1} via the well known extension theorem for the fractional Laplacian $(-\Delta)^\sg$ established by Caffarelli-Silvestre \cite{C-S}. We use capital letters, such as $X=(x, t)\in \mathbb{R}^n \times \mathbb{R}_+$,  to  denote  points in $\mathbb{R}_+^{n+1}$.  We also denote $\mathcal{B}_R$ as  the ball in $\mathbb{R}^{n+1}$ with radius $R$ and center at the origin,  $\mathcal{B}_R^+$ as the upper half-ball $\mathcal{B}_R\cap \mathbb{R}_+^{n+1}$,  and $\partial^0 \mathcal{B}_R^+$ as the flat part of $\partial \mathcal{B}_R^+$  which is the ball $B_R$ in $\mathbb{R}^{n}$.  Then the problem \eqref{Iso1} is equivalent to the following extension problem 
\begin{equation}\label{Iso3}
\begin{cases}
-\textmd{div}(t^{1-2\sg}  \nabla U)=0~~~~~~~~~& \textmd{in} ~ \mathcal{B}_1^+,\\
\frac{\partial U}{\partial \nu^\sg}(x, 0)=U^p(x, 0)~~~~~~~~& \textmd{on} ~ \partial^0\mathcal{B}_1^+ \backslash  \{0\},
\end{cases}
\end{equation}
where $\frac{\partial U}{\partial \nu^\sg}(x, 0) := -\lim_{t\rightarrow 0^+}t^{1-2\sg} \partial_t U(x, t)$.  By \cite{C-S}, we only need to analyze the behavior of the traces
$$
u(x):=U(x, 0)
$$
of the nonnegative solutions $U(x,t)$ of \eqref{Iso3} near the origin, from which we can get the behavior of solutions of  \eqref{Iso1} near the origin.

In order to prove Theorem \ref{THM01}, we need to establish the following cylindrically symmetric result for the global equation that the origin is a non-removable isolated singularity. 
\begin{theorem}\label{THM02}
Let $U$ be a nonnegative solution of
\begin{equation}\label{Iso4}
\begin{cases}
-\textmd{div}(t^{1-2\sg}  \nabla U)=0~~~~~~~~~& \textmd{in} ~ \mathbb{R}^{n+1}_+,\\
\frac{\partial U}{\partial \nu^\sg}(x, 0)=U^p(x, 0)~~~~~~~~& \textmd{on} ~ \mathbb{R}^n \backslash  \{0\},
\end{cases}
\end{equation}
with $\frac{n}{n-2\sg} < p < \frac{n+2\sg}{n-2\sg}$. Assume that the origin 0 is a non-removable singularity. Then $U(x, t)=U(|x|, t)$. 
\end{theorem}
The paper is organized as follow. In Section 2, we prove that singular positive solutions of \eqref{Iso4} are cylindrically symmetric via the method of moving sphere introduced by Li and Zhu \cite{L-Z}. We mainly follow the argument in \cite{C-J-S-X} where the cylindrical symmetry of singular positive solutions of \eqref{Iso4} with $p=\frac{n+2\sg}{n-2\sg}$ was proved.  Section 3 is devoted to the proof  of Theorem \ref{THM01}. We will see that our proof is very different from those in papers \cite{G-S,C-G-S}. In particular, the monotonicity formula established in \cite{Y-Z} and the blow up argument introduced in \cite{G-K-S} are two essential tools. 

\section{Cylindrical Symmetry}
For each $\bar{x} \in \mathbb{R}^n$ and $\lambda >0$, we denote $\overline{X}= (\bar{x}, 0)$ and define the Kelvin transformation of $U$ with respect to the ball $\mathcal{B}_\lambda(\overline{X})$ as follow 
\begin{equation}\label{CS00}
U_{\overline{X}, \lambda} (\xi):=\left( \frac{\lambda}{|\xi - \overline{X}|} \right)^{n-2\sg} U \left(\overline{X} + \frac{\lambda^2(\xi- \overline{X})}{|\xi - \overline{X}|^2}\right). 
\end{equation}
If $U$ is a solution of \eqref{Iso4}, then $U_{\overline{X}, \lambda}$ satisfies
\begin{equation}\label{CS01}
\begin{cases}
-\textmd{div}(t^{1-2\sg}  \nabla U_{\overline{X}, \lambda})=0~~~~~~~~~& \textmd{in} ~ \mathbb{R}^{n+1}_+,\\
\frac{\partial U_{\overline{X}, \lambda}}{\partial \nu^\sg}(y, 0)=\left(\frac{\lambda}{|y - \bar{x}|}\right)^{p^*}U_{\overline{X}, \lambda}^p(y, 0)~~~~~~~~& \textmd{on} ~ \mathbb{R}^n \backslash  \{\bar{x}, y_0\},
\end{cases}
\end{equation}
where $y_0=\bar{x} - \frac{\lambda^2 \bar{x}}{|\bar{x}|^2}$ and $p^*= n+2\sg -p(n-2\sg)>0$. 

\vskip0.20in

\noindent {\it Proof of Theorem \ref{THM02}.}  
Since the origin 0 is a non-removable singularity, by Corollary 3.1 in \cite{Y-Z}, 
$$
\lim_{|\xi| \rightarrow 0} U(\xi) =+\infty. 
$$
{\it Claim 1:} For all $x \in \mathbb{R}^n \backslash\{0\}$ there exists $\lambda_3(x) \in (0, |x|)$ such that for all $0 < \lambda < \lambda_3(x)$ we have 
\begin{equation}\label{CS02}
U_{X,\lambda}(\xi) \leq U(\xi)~~~~~~ \forall ~ |\xi -  X|\geq \lambda, ~~ \xi\neq 0,
\end{equation}
where $X=(x, 0)$ and $U_{X,\lambda}$ is the Kelvin transformation of $U$ with respect to $\mathcal{B}_\lambda(X)$.  The proof of Claim 1 consists of two steps.

{\it Step 1.} We prove that there exist $0 < \lambda_1 < \lambda_2 < |x|$ such that
$$
U_{X,\lambda}(\xi) \leq U(\xi),~~~~~~~~\forall ~ 0< \lambda < \lambda_1,~~~ \lambda < |\xi - X| < \lambda_2.
$$
For every $0 < \lambda < \lambda_1 < \lambda_2$, $\xi \in \partial^+\mathcal{B}_{\lambda_2}^+(X)$, we have $X + \frac{\lambda^2(\xi - X)}{|\xi - X|^2} \in \mathcal{B}_{\lambda_2}^+(X)$. Hence we can choose $\lambda_1=\lambda_1(\lambda_2)$ small such that
$$
\aligned
U_{X,\lambda}(\xi) & =\left( \frac{\lambda}{|\xi - X|} \right)^{n-2\sg} U \left(X + \frac{\lambda^2(\xi- X)}{|\xi - X|^2}\right) \\
& \leq \left(\frac{\lambda_1}{\lambda_2}\right)^{n-2\sg}\sup_{\overline{\mathcal{B}_{\lambda_2}^+(X)}}U \leq \inf_{\partial^+\mathcal{B}_{\lambda_2}^+(X)}U \leq U(\xi).
\endaligned
$$
Thus we have 
$$
U_{X,\lambda}(\xi) \leq U(\xi) ~~~~~~~~ \textmd{on} ~ \partial^+( \mathcal{B}_{\lambda_2}^+(X) \backslash \mathcal{B}_{\lambda}^+(X) )
$$
for all $\lambda_2 >0$ and $0 < \lambda < \lambda_1(\lambda_2)$.

Now we show that $U_{X,\lambda}(\xi) \leq U(\xi)$ in $\mathcal{B}_{\lambda_2}^+(X) \backslash \mathcal{B}_{\lambda}^+(X)$ if $\lambda_2$ is small and $0 < \lambda < \lambda_1(\lambda_2)$.  Because $U_{X,\lambda}$ satisfies \eqref{CS01}, we have
\begin{equation}\label{CS03}
\begin{cases}
-\textmd{div}(t^{1-2\sg}  \nabla ( U_{X, \lambda} - U))=0  & \textmd{in} ~ \mathcal{B}_{\lambda_2}^+(X) \backslash \overline{\mathcal{B}_{\lambda}^+(X)},\\
\frac{\partial }{\partial \nu^\sg}(U_{X, \lambda}- U) (y, 0)=\left(\frac{\lambda}{|y - x|}\right)^{p^*}U_{X, \lambda}^p(y, 0) - U^p(y, 0)  & \textmd{on} ~ \partial^0(\mathcal{B}_{\lambda_2}^+(X) \backslash \overline{\mathcal{B}_{\lambda}^+(X)}). 
\end{cases}
\end{equation}
Let $(U_{X, \lambda} - U)^+:=\max(0, U_{X, \lambda} - U)$ which equals to 0 on $\partial^+( \mathcal{B}_{\lambda_2}^+(X) \backslash \mathcal{B}_{\lambda}^+(X) )$. Multiplying \eqref{CS03} by $(U_{X, \lambda} - U)^+$ and integrating by parts,  and using the narrow domain technique from \cite{B-N}, we obtain
$$
\aligned
& \int_{\mathcal{B}_{\lambda_2}^+(X) \backslash \mathcal{B}_{\lambda}^+(X)} t^{1-2\sg} |\nabla(U_{X, \lambda} - U)^+|^2 \\
&~~ = \int_{B_{\lambda_2}(X) \backslash B_\lambda(X)} \left[\left(\frac{\lambda}{|y - x|}\right)^{p^*}U_{X, \lambda}^p(y, 0) - U^p(y, 0) \right] (U_{X, \lambda} - U)^+ \\
&~~ \leq \int_{B_{\lambda_2}(X) \backslash B_\lambda(X)} (U_{X, \lambda}^p(y, 0) - U^p(y, 0))  (U_{X, \lambda} - U)^+ \\
&~~ \leq C \int_{B_{\lambda_2}(X) \backslash B_\lambda(X)} \left( (U_{X, \lambda} - U)^+ \right)^2 U_{X, \lambda}^{p-1}(y, 0) \\
&~~ \leq C  \left( \int_{B_{\lambda_2}(X) \backslash B_\lambda(X)} \left( (U_{X, \lambda} - U)^+ \right)^{\frac{2n}{n-2\sg}} \right)^{\frac{n-2\sg}{n}}
\left( \int_{B_{\lambda_2}(X) \backslash B_\lambda(X)} U_{X, \lambda}^{\frac{(p-1)n}{2\sg}} \right)^{\frac{2\sg}{n}} \\
&~~ \leq C \left( \int_{\mathcal{B}_{\lambda_2}^+(X) \backslash \mathcal{B}_{\lambda}^+(X)} t^{1-2\sg} |\nabla(U_{X, \lambda} - U)^+|^2 \right)
\left( \int_{B_{\lambda_2}(X)} U^{\frac{(p-1)n}{2\sg}} \right)^{\frac{2\sg}{n}}
\endaligned
$$
We can choose $\lambda_2$ small such that
$$
C \left( \int_{B_{\lambda_2}(X)} U^{\frac{(p-1)n}{2\sg}} \right)^{\frac{2\sg}{n}} \leq \frac{1}{2}.
$$
Then we have
$$
\nabla(U_{X, \lambda} - U)^+ =0 ~~~~~ \textmd{in} ~ \mathcal{B}_{\lambda_2}^+(X) \backslash \mathcal{B}_{\lambda}^+(X). 
$$
Since
$$
(U_{X, \lambda} - U)^+ =0 ~~~~~~~\textmd{on}~  \partial^+( \mathcal{B}_{\lambda_2}^+(X) \backslash \mathcal{B}_{\lambda}^+(X) ), 
$$
we obatin 
$$
(U_{X, \lambda} - U)^+ =0~~~~~~\textmd{in} ~ \mathcal{B}_{\lambda_2}^+(X) \backslash \mathcal{B}_{\lambda}^+(X). 
$$
Hence, we have
$$
U_{X, \lambda} \leq  U~~~~~~\textmd{in} ~ \mathcal{B}_{\lambda_2}^+(X) \backslash \mathcal{B}_{\lambda}^+(X)
$$
for $0 < \lambda < \lambda_1$.

{\it Step 2.}  We prove that there exists $\lambda_3(x) \in (0, \lambda_1)$ such that for each $0 < \lambda < \lambda_3(x)$, 
$$
U_{X, \lambda}(\xi) \leq  U(\xi),~~~~~~~~\forall ~ |\xi - X| > \lambda_2,~~ \xi \neq 0. 
$$
To prove this step, we let 
$$
\phi(\xi)=\left(\frac{\lambda_2}{|\xi - X|}\right)^{n-2\sg} \inf_{\partial^+ \mathcal{B}_{\lambda_2}^+(X)} U, 
$$
then $\phi$ satisfies
$$
\begin{cases}
-\textmd{div}(t^{1-2\sg}  \nabla \phi)=0~~~~~~~~~& \textmd{in} ~ \mathbb{R}^{n+1}_+ \backslash \mathcal{B}_{\lambda_2}^+(X),\\
\frac{\partial \phi}{\partial \nu^\sg}(x, 0)=0~~~~~~~~& \textmd{on} ~ \mathbb{R}^n \backslash  B_{\lambda_2}(X),
\end{cases}
$$
and $\phi(\xi) \leq U(\xi)$ on $\partial^+ \mathcal{B}_{\lambda_2}(X)$. By the maximum principle, we have
$$
U(\xi) \geq \left(\frac{\lambda_2}{|\xi - X|}\right)^{n-2\sg} \inf_{\partial^+ \mathcal{B}_{\lambda_2}^+(X)} U~~~~~~~\forall ~ |\xi - X| \geq \lambda_2,~~ \xi \neq 0. 
$$
Let 
$$
\lambda_3(x):= \min\left\{\lambda_1, \lambda_2\left( \inf_{\partial^+ \mathcal{B}_{\lambda_2}^+(X)} U / \sup_{\mathcal{B}_{\lambda_2}^+(X)} U\right)^{\frac{1}{n-2\sg}} \right\}. 
$$
Then for any $0 < \lambda < \lambda_3(x)$, $|\xi - X| \geq \lambda_2$ and $\xi \neq 0$, we have
$$
\aligned
U_{X,\lambda}(\xi)  & =\left( \frac{\lambda}{|\xi - X|} \right)^{n-2\sg} U \left(X + \frac{\lambda^2(\xi- X)}{|\xi - X|^2}\right) \\
& \leq \left( \frac{\lambda_3}{|\xi - X|} \right)^{n-2\sg} \sup_{\mathcal{B}_{\lambda_2}^+(X)} U \\
& \leq \left( \frac{\lambda_2}{|\xi - X|} \right)^{n-2\sg} \inf_{\partial^+ \mathcal{B}_{\lambda_2}^+(X)} U \leq U(\xi). 
\endaligned
$$
The proof of Claim 1 is completed.

Now, we define
$$
\bar{\lambda}(x):= \sup\{ 0 < \mu \leq |x|  \big| U_{X,\lambda}(\xi) \leq U(\xi), ~~ \forall ~ |\xi -  X|\geq \lambda, ~~ \xi\neq 0, ~~\forall ~0< \lambda < \mu \}.
$$
By Claim 1, $\bar{\lambda}(x)$ is well defined. 

{\it Claim 2.} $\bar{\lambda}(x) = |x|$.

Suppose by contradiction that $\bar{\lambda}(x) < |x|$ for some $x\neq 0$.  Since the origin 0 is not removable, by the strong maximum principle, we obtain
$$
U(\xi) > U_{X,\lambda}(\xi) ~~~~~\forall ~ |\xi -  X| > \lambda, ~~ \xi\neq 0. 
$$
Moreover, we have
$$
\liminf_{\xi \rightarrow 0}(U(\xi)- U_{X,\lambda}(\xi)) >0. 
$$
We can using the narrow domain technique as above, see also the proof of Theorem 1.8 in \cite{J-L-X}. Then there exists $\epsilon_1 >0$ such that for all $\bar{\lambda}(x) < \lambda < \bar{\lambda}(x) +\epsilon_1$ we have 
$$
U_{X,\lambda}(\xi) \leq U(\xi) , ~~~~~~~ \forall ~|\xi - X| \geq \lambda,~~ \xi \neq 0, 
$$
which contradicts with the definition of $\bar{\lambda}(x)$. This proves Claim 2.

Therefore, we obatin 
$$
U_{X,\lambda}(\xi) \leq U(\xi)~~~~~~ \forall ~ |\xi -  X|\geq \lambda, ~~ \xi\neq 0,~~\forall ~ 0< \lambda < |x|. 
$$
In particular,  we have
$$
u_{x,\lambda}(y) \leq u(y)~~~~~~ \forall ~ |y- x |\geq \lambda, ~~ y\neq 0,~~\forall ~ 0< \lambda < |x|,
$$
where $u_{x,\lambda}$ is the Kelvin transformation of $u$ with respect to the ball $B_{\lambda} (x)$. Thus, we deduce from Lemma 2.1 in \cite{J-L-Xu} that  $u$ is radially symmetric about the origin 0.   The proof of Theorem \ref{THM02} is complete. 
\hfill$\square$

\section{Exact Asymptotic Behavior}
In this section, we shall prove Theorem \ref{THM01}. First, we recall the energy functional in \cite{Y-Z} 
$$
\aligned
E(r;U):= &  r^{2\frac{(p+1)\sg}{p-1} - n }\left[ r \int_{\partial^+\mathcal{B}_r^+} t^{1-2\sg} | \frac{\partial U}{\partial \nu} |^2 + \frac{2\sg}{p-1} \int_{\partial^+\mathcal{B}_r^+} t^{1-2\sg}\frac{\partial U}{\partial \nu} U\right] \\
& + \frac{1}{2}\frac{2\sg}{p-1}\left(\frac{4\sg}{p-1} - (n-2\sg)\right) r^{2\frac{(p+1)\sg}{p-1} - n -1}  \int_{\partial^+\mathcal{B}_r^+} t^{1-2\sg} U^2  \\
& -  r^{2\frac{(p+1)\sg}{p-1} - n+1 }  \left[ \frac{1}{2} \int_{\partial^+\mathcal{B}_r^+} t^{1-2\sg} |\nabla U|^2 -\frac{1}{p+1}\int_{\partial B_r} u^{p+1} \right].
\endaligned
$$
We define the scaling function
$$
U^\lambda(X):=\lambda^{\frac{2\sg}{p-1}} U(\lambda X),~~~~~~\lambda >0. 
$$
Then we easily see that the equation \eqref{Iso3} is invariant  under this scaling. More precisely, if $U$ is a solution of \eqref{Iso3} in $\mathcal{B}_R^+ \backslash \{0\}$, then $U^\lambda$ is a solution of \eqref{Iso3} in $\mathcal{B}_{R/\lambda}^+ \backslash \{0\}$. Moreover, we easily check that $E$ satisfies the following scaling relation
\begin{equation}\label{EN01}
E(\lambda s, U)= E(s, U^\lambda), 
\end{equation}
for $\lambda, s>0$.  We remark that this scaling invariance of $E$ plays a key role in the proof of Proposition 3.3 in \cite{Y-Z}. By Proposition 3.2 in \cite{Y-Z} (or more precisely, and its proof there), we have the following monotonicity formula.  
\begin{proposition}\label{EAB302}
Let $U$ be a nonnegative solution of \eqref{Iso3} in $\mathcal{B}_R^+ \backslash \{0\}$ with $1 < p < \frac{n+2\sg}{n-2\sg}$. Then, $E(r; U)$ is non-decreasing in $r \in (0, R)$. Furthermore,
$$
\frac{d}{dr} E(r; U)=J_1 r^{2\frac{(p+1)\sg}{p-1} -n} \int_{\partial^+\mathcal{B}_r^+} t^{1-2\sg} \left(\frac{\partial U}{\partial \nu} + \frac{2\sg}{p-1} \frac{U}{r}\right)^2,
$$
where $J_1=\frac{4\sg}{p-1} - (n-2\sg) >0$ since $1 < p < \frac{n+2\sg}{n-2\sg}$.  In particular, if $E(r; U) $ is constant for $r\in (R_1, R_2)$, then $U$ is homogeneous of degree $-\frac{2\sg}{p-1}$ in $\mathcal{B}_{R_2}^+ \backslash \mathcal{B}_{R_1}^+$, i.e.
$$
U(X)=|X|^{-\frac{2\sg}{p-1}}U \left( \frac{X}{|X|} \right). 
$$
\end{proposition}

\begin{proposition}\label{EAB303}
Let $U$ be a nonnegative solution of \eqref{Iso3} in $\mathcal{B}_1^+ \backslash \{0\}$ with $1 < p < \frac{n+2\sg}{n-2\sg}$. Then $E(r; U)$ is uniformly bounded for $0< r < \frac{1}{2}$. Furthermore, the limit 
$$
E(0^+; U):=\lim_{r\rightarrow 0^+} E(r; U)
$$
exists. 
\end{proposition}
\begin{proof}
Let
$$
V(X)=r^{\frac{2\sg}{p-1}}U(rX)
$$
for any $r\in (0, \frac{1}{2})$ and $\frac{1}{2} \leq |X|\leq 2$. Then $V$ satisfies
$$
\begin{cases}
-\textmd{div}(t^{1-2\sg}  \nabla V)=0~~~~~~~~~& \textmd{in} ~\mathcal{B}_2 \backslash \overline{\mathcal{B}}_{1/2},\\
\frac{\partial V}{\partial \nu^\sg}(x, 0)=v^p(x)~~~~~~~~& \textmd{on} ~B_2 \backslash \overline{B}_{1/2},\\
\end{cases}
$$
where $v(x)=V(x, 0)$. It follows from Proposition 3.1 in \cite{Y-Z} that 
$$
|V(x)| \leq C ~~~~~~~~ \textmd{for} ~ \textmd{all} ~ \frac{1}{2} \leq |X| \leq 2,
$$
where the constant $C$ depend only on $n, p, \sg$. By Proposition 2.19 in \cite{J-L-X}, we have
$$
\sup_{\frac{3}{4} \leq |X| \leq \frac{3}{2}} |\nabla_x V| + \sup_{\frac{3}{4} \leq |X| \leq \frac{3}{2}} | t^{1-2\sg}\partial_t V| \leq C,
$$
where the constant $C$  depend only on $n, p, \sg$. Hence, there exists $C>0$ depend only on $n, p, \sg$, such that
$$
|\nabla_x U (X)|\leq C |X|^{-\frac{2\sg}{p-1} -1}~~~~~~~~ \textmd{in} ~ \mathcal{B}_{1/2}^+ \backslash \{0\}
$$
and 
$$
|t^{1-2\sg}\partial_t U(X)|\leq C |X|^{-\frac{2\sg}{p-1} -2\sg}~~~~~~~~ \textmd{in} ~ \mathcal{B}_{1/2}^+ \backslash \{0\}. 
$$
Thus, we can estimate
$$
\aligned
r^{2\frac{(p+1)\sg}{p-1}-n+1} \int_{\partial^+ \mathcal{B}_r^+} t^{1-2\sg} |\nabla U| \leq & C r^{2\frac{(p+1)\sg}{p-1}-n+1} \bigg( r^{-\frac{4\sg}{p-1}-2} \int_{\partial^+ \mathcal{B}_r^+} t^{1-2\sg} \\
& + r^{-\frac{4\sg}{p-1}-4\sg} \int_{\partial^+ \mathcal{B}_r^+} t^{2\sg-1}\bigg) \leq C,
\endaligned
$$
$$
r^{2\frac{(p+1)\sg}{p-1}-n-1} \int_{\partial^+ \mathcal{B}_r^+} t^{1-2\sg} U^2 \leq C r^{2\sg-n-1}\int_{\partial^+ \mathcal{B}_r^+} t^{1-2\sg} \leq C
$$
and 
$$
r^{2\frac{(p+1)\sg}{p-1}-n+1}\int_{\partial B_r} u^{p+1} \leq C,
$$
where the constant $C$ also depend only on $n, p, \sg$. Now  we easily conclude that $E(r; U)$ is uniformly bounded for $0< r < \frac{1}{2}$. By the monotonicity of $E(r; U)$, we obtain the limit 
$$\lim_{r\rightarrow 0^+} E(r; U)$$
exists. 
\end{proof}

\begin{proposition}\label{EAB304}
Let $U$ be a nonnegative solution of \eqref{Iso4} with $\frac{n}{n-2\sg} < p< \frac{n+2\sg}{n-2\sg}$. Suppose that $U$ is homogeneous of degree $-\frac{2\sg}{p-1}$. Then either $U\equiv 0$ in $\overline{\mathbb{R}^{n+1}_+}$, or the trace $u(x)=U(x, 0)$ of $U$  is of the form
$$
u(x)=A_{n,p,\sg}|x|^{-\frac{2\sg}{p-1}}, 
$$
where $A_{n,p,\sg}$ is given by \eqref{A}. 
\end{proposition}
\bp
Suppose that $U$ is nontrivial solution, then by strong maximum principle, 
$$U(x, t) >0 ~~~~~~~\textmd{in} ~~ \mathbb{R}^{n+1}_+ \backslash \{0\}.$$
Hence,  by the Liouville type theorem in \cite{J-L-X}, the origin 0 must be a non-removable singularity.  By Theorem \ref{THM02}, $u(x)$ is radially symmetric, hence $u$ is a positive constant $a$ on $\partial B_1$.  
By the homogeneity of $u$, we have
$$
u(x)=a|x|^{-\frac{2\sg}{p-1}}. 
$$
On the other hand, since $u$ satisfies
$$
(-\Delta)^\sigma u = u^p~~~~~\textmd{in}~ \mathbb{R}^n\backslash \{0\}. 
$$
By a classical calculation,  see, for instance, Lemma 3.1 in Fall \cite{Fall}, we obtain
$$
a=A_{n,p,\sg}. 
$$
This finishes the proof. 
\ep

\noindent {\it Proof of Theorem \ref{THM01}.}   Suppose that the origin 0 is a non-removable singularity, we only need to prove \eqref{Be01}. Consider the scaling
$$
U^\lambda(X)=\lambda^{\frac{2\sg}{p-1}} U(\lambda X). 
$$
Then $U^\lambda$ satisfies
$$
\begin{cases}
-\textmd{div}(t^{1-2\sg}  \nabla U^\lambda)=0~~~~~~~~~& \textmd{in} ~ \mathcal{B}_{1/\lambda}^+,\\
\frac{\partial U^\lambda}{\partial \nu^\sg}(x, 0)=(U^\lambda(x, 0))^p~~~~~~~~& \textmd{on} ~ \partial^0\mathcal{B}_{1/\lambda}^+ \backslash  \{0\}. 
\end{cases}
$$
By \eqref{Be00} and Remark 1.2 in \cite{Y-Z}
\begin{equation}\label{TA01}
C_1 |X|^{-\frac{2\sg}{p-1}}\leq U^\lambda (X) \leq C_2 |X|^{-\frac{2\sg}{p-1}}~~~~~\textmd{in} ~ \mathcal{B}_{1/\lambda}^+. 
\end{equation}
Thus $U^\lambda$ is locally uniformly bounded away from the origin. By Corollary 2.10 and Theorem 2.15  in \cite{J-L-X} that  there exists $\alpha >0$ such that for every $R > 1 > r > 0$
$$
\|U^\lambda \|_{W^{1,2}(t^{1-2\sg}, \mathcal{B}_R^+ \backslash \overline{\mathcal{B}}_r^+)} + \|U^\lambda\|_{C^\alpha (\mathcal{B}_R^+ \backslash \overline{\mathcal{B}}_r^+)} + \|u^\lambda\|_{C^{2,\alpha}(B_R \backslash B_r)} \leq C(R, r),
$$
where $u^\lambda(x)=U^\lambda(x, 0)$ and $C(R, r)$ is independent of $\lambda$.  Then there is a subsequence $\lambda_k \rightarrow 0$,  $\{U^{\lambda_k}\}$ converges to a nonnegative function $U^0  \in W_{loc}^{1,2}(t^{1-2\sg}, \overline{\mathbb{R}^{n+1}_+} \backslash \{0\}) \cap C^\alpha_{loc}(\overline{\mathbb{R}^{n+1}_+} \backslash \{0\})$ satisfying 
$$
\begin{cases}
-\textmd{div}(t^{1-2\sg}  \nabla U^0)=0~~~~~~~~~& \textmd{in} ~\mathbb{R}^{n+1}_+,\\
\frac{\partial U^0}{\partial \nu^\sg}(x, 0)=(U^0(x,0))^p~~~~~~~~& \textmd{on} ~\mathbb{R}^n \backslash \{0\}, 
\end{cases}
$$
and by \eqref{TA01} we have 
\begin{equation}\label{TA02}
C_1 |X|^{-\frac{2\sg}{p-1}} \leq U^0(X) \leq C_2 |X|^{-\frac{2\sg}{p-1}}~~~~~~\textmd{in} ~\mathbb{R}^{n+1}_+\backslash \{0\}. 
\end{equation}
Moreover, since the scaling relation \eqref{EN01}, we have, for any $r>0$,
$$
E(r; U^0)= \lim_{k\rightarrow \infty}E(r; U^{\lambda_k})  = \lim_{k\rightarrow \infty}E(r\lambda_k; U)=E(0^+; U).
$$
Thus, by Proposition 3.1, $U^0$ is homogeneous of degree $-\frac{2\sg}{p-1}$. It follows from Proposition 3.3 and \eqref{TA02} that the trace $u^0(x)=U^0(x, 0)$  must be of the form
$$
u^0(x)=A_{n,p,\sg}|x|^{-\frac{2\sg}{p-1}}, 
$$
where $A_{n,p,\sg}$ is given by \eqref{A}. Since the function $u^0(x)$ is unique, we conclude that $u^\lambda(x)\rightarrow u^0(x)$ for any sequence $\lambda\rightarrow 0$ in $C^\alpha_{loc}(\mathbb{R}^n \backslash \{0\})$.  Hence
$$
|\lambda x|^{\frac{2\sg}{p-1}} u(\lambda x)= |x|^{\frac{2\sg}{p-1}} u^\lambda(x)\rightarrow A_{n,p,\sg}~~~~~ \textmd{as}~~ \lambda \rightarrow 0
$$
in $B_2 \backslash B_{1/2}$. We immediately conclude that
$$
\lim_{|x|\rightarrow 0} |x|^{\frac{2\sigma}{p-1}} u(x) = A_{n,p,\sigma}.
$$
\hfill$\square$

\end{document}